\documentclass{amsart}
\usepackage{bm}
\usepackage{cases}
\usepackage[all]{xy}
\usepackage{amsmath}
\usepackage{amssymb}
\usepackage{amscd}
\usepackage{amsthm}
\usepackage[dvips]{graphicx}
\usepackage{color}

\theoremstyle{definition}
\newtheorem{defn}{Definition}[section]

\newtheorem{rem}[defn]{Remark}
\theoremstyle{plain}
\newtheorem{thm}[defn]{Theorem}
\newtheorem{prop}[defn]{Proposition}
\newtheorem{lem}[defn]{Lemma}
\newtheorem{cor}[defn]{Corollary}

\numberwithin{equation}{section}
\allowdisplaybreaks
\title[Generators for the handlebody subgroup of the Torelli group]{A small normal generating set for the handlebody subgroup of the Torelli group}

\author[G.~Omori]{Genki Omori}
\address{
(Genki Omori)
Department of Mathematics,
Tokyo Institute of Technology,
Oh-okayama, Meguro, Tokyo 152-8551, Japan
}

\email{omori.g.aa@m.titech.ac.jp}
\date{\today}
\begin{document}
\maketitle
\begin{abstract}
We prove that the handlebody subgroup of the Torelli group of an orientable surface is generated by genus one BP-maps
. As an application, we give a normal generating set for the handlebody subgroup of the level $d$ mapping class group of an orientable surface.
\end{abstract}

\section{Introduction}

Let $H_g$ be an oriented 3-dimensional handlebody of genus $g$ and let $D_0$ be a disk on the boundary $\Sigma _g=\partial H_g$ of $H_g$. We fix a model of $H_g$ and $D_0$ as in Figure~\ref{handlebody_scc1} and set $\Sigma _{g,1}:=\Sigma _g-{\rm int}D_0$. The {\it mapping class group} $\mathcal{M}_{g,1}$ of $\Sigma _{g,1}$ is the group of isotopy classes of orientation preserving self-diffeomorphisms on $\Sigma _g$ fixing $D_0$ pointwise and the {\it handlebody group} $\mathcal{H}_{g,1}$ is the subgroup of $\mathcal{M}_{g,1}$ which consists of elements that extend to $H_g$.

For a simple closed curve $c$ on $\Sigma _{g,1}$, denote by $t_c$ the right-handed Dehn twist along $c$. A pair $\{ c_1, c_2\}$ of simple closed curves $c_1$ and $c_2$ on $\Sigma _{g,1}$ is a {\it bounding pair (BP) on} $\Sigma _{g,1}$ if $c_1$ and $c_2$ are disjoint, non-isotopic and their integral homology classes are non-trivial and the same. A BP $\{ c_1,c_2\}$ on $\Sigma _{g,1}$ is a {\it genus-$h$ bounding pair (genus-$h$ BP) on} $\Sigma _{g,1}$ if the union of $c_1$ and $c_2$ bounds a subsurface of $\Sigma _{g,1}$ of genus $h$ with two boundary components. For a BP (resp. genus-$h$ BP) $\{ c_1,c_2\}$ on $\Sigma _{g,1}$, we call $t_{c_1}t_{c_2}^{-1}$ a {\it BP-map} (resp. {\it genus-$h$ BP-map}). 

The {\it Torelli group} $\mathcal{I}_{g,1}$ of $\Sigma _{g,1}$ is the the kernel of a homomorphism $\Psi :\mathcal{M}_{g,1}\rightarrow {\rm Sp}(2g,\mathbb Z)$ induced by the action of $\mathcal{M}_{g,1}$ on the integral first homology group ${\rm H}_1(\Sigma _{g,1};\mathbb Z)$ of $\Sigma _{g,1}$. Genus-$h$ BP-maps are elements of $\mathcal{I}_{g,1}$. For a group $G$, a normal subgroup $H$ of $G$ and elements $x_1, x_2, \dots ,x_n$ of $H$, $H$ is {\it normally generated in $G$ by $x_1, x_2, \dots ,x_n$} if $H$ is the normal closure of $\{ x_1, x_2, \dots ,x_n\}$ in $G$. By an argument of Powell~\cite{Powell}, $\mathcal{I}_{g,1}$ is normally generated in $\mathcal{M}_{g,1}$ by a genus-1 BP-map and Dehn twists along separating simple closed curves (actually, Powell proved that the Torelli group of an closed oriented surface is generated by genus-1 BP-maps and Dehn twists along separating simple closed curves by using Birman's finite presentation~\cite{Birman2} for the symplectic group ${\rm Sp}(2g,\mathbb Z)$). Johnson showed that $\mathcal{I}_{g,1}$ is normally generated in $\mathcal{M}_{g,1}$ by a genus-1 BP-map in~\cite{Johnson1} and gave an explicit finite generating set for $\mathcal{I}_{g,1}$ in~\cite{Johnson2}. A smaller finite generating set for $\mathcal{I}_{g,1}$ is given by Putman~\cite{Putman1}. 

Denote by $\mathcal{V}(3)$ the set of diffeomorphism classes of connected closed oriented 3-manifolds and by $\mathcal{S}(3)$ the set of diffeomorphism classes of integral homology 3-spheres. Let $H_g^\prime $ be a 3-dimensional handlebody of genus $g$ such that $\partial H_g^\prime =\Sigma _g$ and the union $H_g\cup H_g^\prime $ is diffeomorphic to the 3-sphere $S^3$, and let $\mathcal{H}_{g,1}^\prime $ be the subgroup of $\mathcal{M}_{g,1}$ which consists of elements that extend to $H_g^\prime $. For each $f\in \mathcal{M}_{g,1}$, we denote by $M_f$ the closed oriented 3-manifold obtained by gluing the disjoint union of $H_g$ and $H_g^\prime $ along $f$. We regard $\mathcal{M}_{g,1}$ as a subgroup of $\mathcal{M}_{g+1,1}$ by a natural injective stabilization map $\mathcal{M}_{g,1}\hookrightarrow \mathcal{M}_{g+1,1}$. Then we have a bijection
\[
\lim_{g \to \infty} \mathcal{H}_{g,1}\setminus \mathcal{M}_{g,1}/\mathcal{H}_{g,1}^\prime \longrightarrow \mathcal{V}(3) 
\]
by $[f]$ to $M_f$ (see for instance~\cite{Birman1}). The above bijection induces the following bijection~\cite{Morita}:
\[
\lim_{g \to \infty} \mathcal{H}_{g,1}\setminus \mathcal{I}_{g,1}/\mathcal{H}_{g,1}^\prime \longrightarrow \mathcal{S}(3). 
\]
Hence any integral homology 3-sphere is represented by an element of $\mathcal{I}_{g,1}$. Note that $\mathcal{H}_{g,1}$ and $\mathcal{H}_{g,1}^\prime $ are not subgroups of $\mathcal{I}_{g,1}$, and for $f$, $h\in \mathcal{I}_{g,1}$, $[f]=[h]\in \mathcal{H}_{g,1}\setminus \mathcal{I}_{g,1}/\mathcal{H}_{g,1}^\prime $ means there exist elements $\varphi \in \mathcal{H}_{g,1}$ and $\varphi ^\prime \in \mathcal{H}_{g,1}^\prime $ such that $h=\varphi f \varphi ^\prime \in \mathcal{I}_{g,1}$. We denote by $\mathcal{IH}_{g,1}$ (resp. $\mathcal{IH}_{g,1}^\prime $) the intersection of $\mathcal{I}_{g,1}$ and $\mathcal{H}_{g,1}$ (resp. $\mathcal{H}_{g,1}^\prime $). Pitsch~\cite{Pitsch} gave the following theorem.
\begin{thm}[\cite{Pitsch}]\label{Pitsch}
For $f$, $h\in \mathcal{I}_{g,1}$, $[f]=[h]\in \mathcal{H}_{g,1}\setminus \mathcal{I}_{g,1}/\mathcal{H}_{g,1}^\prime $ if and only if there exist elements $\varphi \in \mathcal{IH}_{g,1}$, $\varphi ^\prime \in \mathcal{IH}_{g,1}^\prime $ and $\psi \in \mathcal{H}_{g,1}\cap \mathcal{H}_{g,1}^\prime $ such that 
\[
h=\psi \varphi f \varphi ^\prime \psi ^{-1}.
\]
\end{thm}
For these reasons, it is important for the classification of integral homology 3-spheres to give a simple generating set for $\mathcal{IH}_{g,1}$.

For a genus-$h$ BP $\{ c_1, c_2\}$ on $\Sigma _{g,1}$, $\{ c_1, c_2\}$ is a {\it genus-$h$ homotopical bounding pair (genus-$h$ HBP)} on $\Sigma _{g,1}$ if each $c_i$ $(i=1,2)$ doesn't bound a disk on $H_g$ and the disjoint union $c_1\sqcup c_2$ bounds an annulus on $H_g$. We remark that such an annulus is unique up to isotopy by the irreducibility of $H_g$. For example, a pair $\{ C_1,C_2\}$ of simple closed curves $C_1$ and $C_2$ on $\Sigma _{g,1}$ as in Figure~\ref{handlebody_scc1} is a genus-1 HBP on $\Sigma _{g,1}$. For a genus-$h$ HBP $\{ c_1, c_2\}$ on $\Sigma _{g,1}$, we call $t_{c_1}t_{c_2}^{-1}$ a {\it genus-$h$ HBP-map}. Hence $t_{C_1}t_{C_2}^{-1}$ is a genus-$h$ HBP-map. Remark that genus-$h$ HBP-maps are elements of $\mathcal{IH}_{g,1}$. The main theorem in this paper is as follows:
\begin{thm}\label{mainthm}
For $g\geq 3$, $\mathcal{IH}_{g,1}$ is normally generated in $\mathcal{H}_{g,1}$ by $t_{C_1}t_{C_2}^{-1}$. 
In particular, for $g\geq 3$, $\mathcal{IH}_{g,1}$ is generated by genus-1 HBP-maps.
\end{thm}
We prove Theorem~\ref{mainthm} in Section~\ref{proof}. In Section~\ref{section_conjugation}, we give a necessary and sufficient condition that a genus-1 HBP-map is conjugate to $t_{C_1}t_{C_2}^{-1}$ in $\mathcal{H}_{g,1}$.

For $d\geq 2$, we define $\mathbb Z_d:=\mathbb Z/d\mathbb Z$. The {\it level} $d$ {\it mapping class group} $\mathcal{M}_{g,1}[d]$ is the kernel of a homomorphism $\Psi_d :\mathcal{M}_{g,1}\rightarrow {\rm Sp}(2g,\mathbb Z_d)$ induced by the action of $\mathcal{M}_{g,1}$ on ${\rm H}_1(\Sigma _{g,1};\mathbb Z_d)$. Denote by $\mathcal{H}_{g,1}[d]$ the intersection of $\mathcal{M}_{g,1}[d]$ and $\mathcal{H}_{g,1}$. Let $D_1, D_2, \dots , D_g$ and $C_2^\prime $ be simple closed curves on $\Sigma _{g,1}$ as in Figure~\ref{handlebody_scc1}. Each of $D_1, D_2, \dots , D_g$ bounds a disk in $H_g$. We define $\alpha :=t_{C_1}t_{C_2^\prime }^{-1}$ and denote by $\omega $ the diffeomorphism on $\Sigma _{g,1}$ which is described as the result of the half rotation of the first handle of $H_g$ as in Figure~\ref{handle_half_rotation}. Note that $\alpha ^d$, $t_{D_i}^d$ $(i=1,\dots ,g)$, and a genus-$h$ HBP-maps are elements of $\mathcal{H}_{g,1}[d]$ and $\omega $ is an element of $\mathcal{H}_{g,1}[2]$. As an application of Theorem~\ref{mainthm}, we obtain the following theorem. The proof is given in Section~\ref{proof_application}.
\begin{thm}\label{thm_application}
For $g\geq 3$, $\mathcal{H}_{g,1}[2]$ is normally generated in $\mathcal{H}_{g,1}$ by $\omega $, $t_{D_1}^2$ and $t_{C_1}t_{C_2}^{-1}$. 

For $g\geq 3$ and $d\geq 3$, $\mathcal{H}_{g,1}[d]$ is normally generated in $\mathcal{H}_{g,1}$ by $\alpha ^d$, $t_{D_1}^d$ and $t_{C_1}t_{C_2}^{-1}$. 
\end{thm}

Let $\mathcal{I}(H_g\text{ rel }D_0)$ (resp. $\Gamma _d(H_g\text{ rel }D_0)$) be the the kernel of the natural homomorphism $\mathcal{H}_{g,1}\rightarrow {\rm Aut\ H}_1(H_g;\mathbb Z)$ (resp. $\mathcal{H}_{g,1}\rightarrow {\rm Aut\ H}_1(H_g;\mathbb Z_d)$). As a corollary of Theorem~\ref{mainthm} and Theorem~\ref{thm_application}, we have the following result.
\begin{cor}\label{thm_cor}
For $g\geq 3$, $\mathcal{I}(H_g\text{ rel }D_0)$ is normally generated in $\mathcal{H}_{g,1}$ by $t_{D_1}$ and $t_{C_1}t_{C_2}^{-1}$. 

For $g\geq 3$, $\Gamma _2(H_g\text{ rel }D_0)$ is normally generated in $\mathcal{H}_{g,1}$ by $\omega $, $t_{D_1}$ and $t_{C_1}t_{C_2}^{-1}$. 

For $g\geq 3$ and $d\geq 3$, $\Gamma _d(H_g\text{ rel }D_0)$ is normally generated in $\mathcal{H}_{g,1}$ by $\alpha ^d$, $t_{D_1}$ and $t_{C_1}t_{C_2}^{-1}$. 
\end{cor}
We prove Corollary~\ref{thm_cor} in Section~\ref{proof_cor}. Luft~\cite{Luft} proved that $\mathcal{I}(H_g\text{ rel }D_0)$ is normally generated in $\mathcal{H}_{g,1}$ by disk twists and a map whose action on the fundamental group of $H_g$ is the same as the action of $t_{C_1}t_{C_2}^{-1}$. An action of $\alpha ^2$ on ${\rm H}_1(\Sigma _{g,1};\mathbb Z)$ is non-trivial, however, an action of a BP-map on ${\rm H}_1(\Sigma _{g,1};\mathbb Z)$ is trivial. As a corollary of Corollary~\ref{thm_cor}, we also have the following corollary. The proof is given in Section~\ref{proof_cor2}.

\begin{cor}\label{thm_cor2}
For $g\geq 3$, $\Gamma _2(H_g\text{ rel }D_0)$ is normally generated in $\mathcal{H}_{g,1}$ by $\omega $, $t_{D_1}$ and $\alpha ^2$.
\end{cor}

\begin{figure}[h]
\includegraphics[scale=0.90]{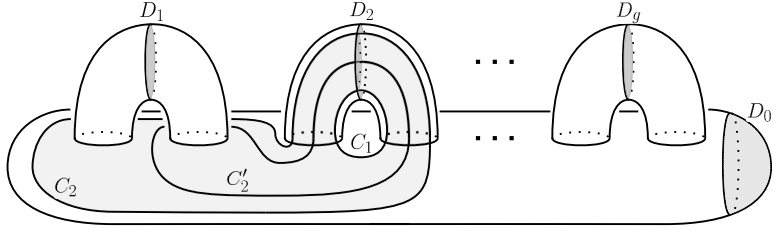}
\caption{The model of $H_g$ and simple closed curves $D_1, D_2, \dots , D_g$, $C_1$, $C_2$ and $C_2^\prime $ on $\Sigma _{g,1}$.}\label{handlebody_scc1}
\end{figure}
\begin{figure}[h]
\includegraphics[scale=0.65]{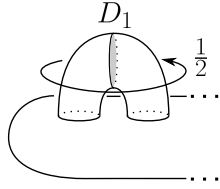}
\caption{The element $\omega $ of $\mathcal{H}_{g,1}$.}\label{handle_half_rotation}
\end{figure}


\section{Generators for the handlebody subgroup of the Torelli group}\label{section_main}

\subsection{Proof of main theorem}\label{proof}
In this section, we prove Theorem~\ref{mainthm}. Let $x_0$ be a point of $\partial D_0$ and let $v_1, v_2, \dots ,v_g$ be generators for the fundamental group $\pi _1(H_g,x_0)$ of $\Sigma _g$ represented by loops on $\Sigma _{g,1}$ based at $x_0$ as in Figure~\ref{handlebody_basis2}. We identify $\pi _1(H_g,x_0)$ with the free group $F_g$ of rank $g$ by the generators. Since $\mathcal{H}_{g,1}$ acts on $\pi _1(H_g,x_0)=F_g$, we have a homomorphism $\eta :\mathcal{H}_{g,1}\rightarrow {\rm Aut}F_g$. Griffiths~\cite{Griffiths} showed that $\eta $ is surjective. Denote by $\mathcal{L}_{g,1}$ the kernel of $\eta $. Luft~\cite{Luft} proved that $\mathcal{L}_{g,1}$ is generated by disk twists. Then we have the exact sequence 
\[
1\longrightarrow \mathcal{L}_{g,1}\longrightarrow \mathcal{H}_{g,1}\stackrel{\eta }{\longrightarrow }{\rm Aut}F_g\longrightarrow 1.
\]
\begin{figure}[h]
\includegraphics[scale=0.7]{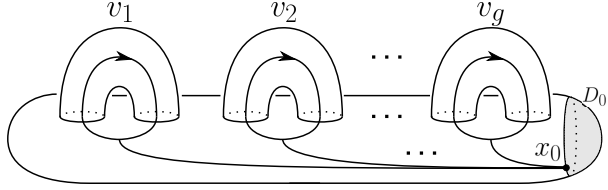}
\caption{Generators $v_1, v_2, \dots ,v_g$ for $\pi _1(H_g,x_0)$.}\label{handlebody_basis2}
\end{figure}
The {\it IA-subgroup} ${\rm IA}_g$ of ${\rm Aut}F_g$ is the kernel of the homomorphism ${\rm Aut}F_g\rightarrow {\rm Aut}\mathbb Z\cong {\rm GL}(g,\mathbb Z)$ induced by the abelianization of $F_g$. Remark that the image $\eta (\mathcal{IH}_{g,1})$ of $\mathcal{IH}_{g,1}$ is included in ${\rm IA}_g$. We define an element $C_{v_1,v_2}$ of ${\rm IA}_g$ by $C_{v_1,v_2}(v_1):=v_2v_1v_2^{-1}$ and $C_{v_1,v_2}(v_k):=v_k$ for $k=2,\dots ,g$. Magnus~\cite{Magnus} proved the following theorem (see also \cite{Day-Putman}).
\begin{thm}[\cite{Magnus}]\label{IA}
For $g\geq 2$, ${\rm IA}_g$ is normally generated in ${\rm Aut}(F_g)$ by $C_{v_1,v_2}$.
\end{thm}

Since $\eta (t_{C_1}t_{C_2}^{-1})=C_{v_1,v_2}$ and $\eta $ is surjective, we have $\eta (\mathcal{IH}_{g,1})={\rm IA}_g$. Denote by $\mathcal{LI}_{g,1}$ the kernel of the homomorphism $\eta |_{\mathcal{IH}_{g,1}}$. $\mathcal{LI}_{g,1}$ is called the {\it Luft-Torelli group} in~\cite{Pitsch}. Then we have the exact sequence
\begin{eqnarray}\label{exact1} 
1\longrightarrow \mathcal{IL}_{g,1}\longrightarrow \mathcal{IH}_{g,1}\stackrel{\eta |_{\mathcal{IH}_{g,1}}}{\longrightarrow }{\rm IA}_g\longrightarrow 1.
\end{eqnarray} 

A BP (resp. genus-$h$ BP) $\{ c_1,c_2\}$ on $\Sigma _{g,1}$ is a {\it contractible bounding pair (CBP)} (resp. {\it genus-$h$ contractible bounding pair (genus-$h$ CBP)}) if each $c_i$ $(i=1,2)$ bounds a disk in $H_g$. For example, $\{ D_2, D_2^\prime \}$ is a genus-1 CBP on $\Sigma _{g,1}$, where $D_2^\prime $ is a simple closed curve on $\Sigma _{g,1}$ as in Figure~\ref{d2prime}. For a CBP (resp. genus-$h$ CBP) $\{ c_1, c_2\}$ on $\Sigma _{g,1}$, we call $t_{c_1}t_{c_2}^{-1}$ a {\it CBP-map} (genus-$h$ CBP-map). CBP-maps are elements of $\mathcal{IL}_{g,1}$. Pitsch~\cite{Pitsch} proved the following theorem.

\begin{figure}[h]
\includegraphics[scale=0.7]{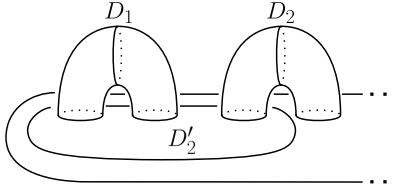}
\caption{Simple closed curve $D_2^\prime $ on $\Sigma _{g,1}$.}\label{d2prime}
\end{figure}

\begin{thm}[\cite{Pitsch}]\label{Luft-Torelli}
For $g\geq 3$, $\mathcal{IL}_{g,1}$ is generated by CBP-maps.
\end{thm}
By Johnson's argument~\cite{Johnson1}, this theorem is improved as follows.
\begin{prop}\label{improvemnt1}
For $g\geq 3$, $\mathcal{IL}_{g,1}$ is normally generated in $\mathcal{H}_{g,1}$ by a genus-1 CBP-map.
\end{prop}
\begin{proof}
Let $\{ c_1,c_2\}$ be a genus-$h$ CBP on $\Sigma _{g,1}$. Without loss of generality, we can assume that each $c_i$ $(i=1,2)$ doesn't intersect with $D_0$. Take proper disks $d_1$ and $d_2$ in $H_g$ such that $\partial d_i=c_i$ for $i=1,2$. By cutting $H_g$ along $d_1\sqcup d_2$, we obtain a handlebody $H$ of genus $h$ which doesn't include $D_0$. Then there exist proper disjoint disks $d_1=e_1, e_2, \dots , e_{h+1}=d_2$ in $H$ such that the result of cutting $H$ along $e_1\sqcup e_2\sqcup \cdots \sqcup e_{h+1}$ is a disjoint union of $h$ handlebodyies of genus $1$, $e_i$ and $e_{i+1}$ lie on a boundary of the same component for $i=1,2,\dots ,h$, and $e_i$ and $e_j$ don't lie on the same component for $|i-j|>1$ (see Figure~\ref{proof_johnson1}). Then we have
\begin{eqnarray*}
t_{c_1}t_{c_2}^{-1}&=&t_{\partial e_1}t_{\partial e_{h+1}}^{-1}\\
&=&(t_{\partial e_1}t_{\partial e_2}^{-1})(t_{\partial e_2}t_{\partial e_3}^{-1})\cdots (t_{\partial e_{h-1}}t_{\partial e_h}^{-1})(t_{\partial e_h}t_{\partial e_{h+1}}^{-1}).
\end{eqnarray*}
Since each $t_{\partial e_i}t_{\partial e_{i+1}}^{-1}$ $(i=1,2,\dots ,h)$ is a genus-$1$ CBP-map, $t_{c_1}t_{c_2}^{-1}$ is a product of genus-$1$ CBP-maps. We get Proposition~\ref{improvemnt1}.
\end{proof}

\begin{figure}[h]
\includegraphics[scale=1.4]{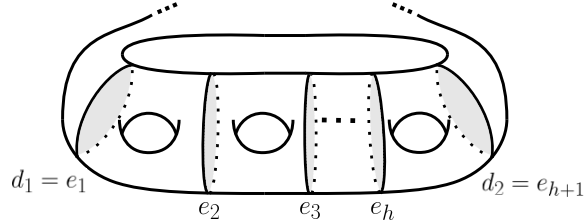}
\caption{Disks $d_1=e_1, e_2, \dots , e_{h+1}=d_2$ on $H$.}\label{proof_johnson1}
\end{figure}

\begin{proof}[Proof of Theorem~\ref{mainthm}]
By the exact sequence~(\ref{exact1}) and Proposition~\ref{improvemnt1}, $\mathcal{IH}_{g,1}$ is normally generated in $\mathcal{H}_{g,1}$ by $t_{C_1}t_{C_2}^{-1}$ and $t_{D_2}t_{D_2^\prime }^{-1}$. Hence it is enough for the proof of Theorem~\ref{mainthm} to show that $t_{D_2}t_{D_2^\prime }^{-1}$ is a product of conjugations of $t_{C_1}t_{C_2}^{-1}$ in $\mathcal{IH}_{g,1}$. Since $(t_{C_1}t_{C_2}^{-1})^{-1}(D_2)=D_2^\prime $, we have
\begin{eqnarray*}
t_{D_2}t_{D_2^\prime }^{-1}&=&t_{D_2}\cdot (t_{C_1}t_{C_2}^{-1})^{-1}t_{D_2}^{-1}(t_{C_1}t_{C_2}^{-1})\\
&=&t_{D_2}(t_{C_1}t_{C_2}^{-1})^{-1}t_{D_2}^{-1}\cdot t_{C_1}t_{C_2}^{-1}.
\end{eqnarray*}
We have completed the proof of Theorem~\ref{mainthm}.
\end{proof}

\begin{rem}\label{pushing_rel}
The last relation 
\[
t_{D_2}t_{D_2^\prime }^{-1}=t_{D_2}(t_{C_1}t_{C_2}^{-1})^{-1}t_{D_2}^{-1}\cdot t_{C_1}t_{C_2}^{-1}
\]
in the proof of Theorem~\ref{mainthm} has the following geometric meaning. Let $E_1$ be a separating disk in $H_g$ as in Figure~\ref{composite_rel1}. Then we can regard $t_{D_2}t_{D_2^\prime }^{-1}$, $t_{C_1}t_{C_2}^{-1}$ and $t_{D_2}(t_{C_1}t_{C_2}^{-1})^{-1}t_{D_2}^{-1}$ as pushing maps of $E_1$ along simple loops on the boundary of the closure of the complement of the first 1-handle. $t_{C_1}t_{C_2}^{-1}$ is obtained from the pushing map along $\gamma _1$ and $t_{D_2}(t_{C_1}t_{C_2}^{-1})^{-1}t_{D_2}^{-1}$ is obtained from the pushing map along $\gamma _2$ as in Figure~\ref{composite_rel1}. The above relation means a product of pushing maps along simple loops which intersect transversely once is equal to the pushing map along the product of these loops.  
\end{rem}

\begin{figure}[h]
\includegraphics[scale=0.80]{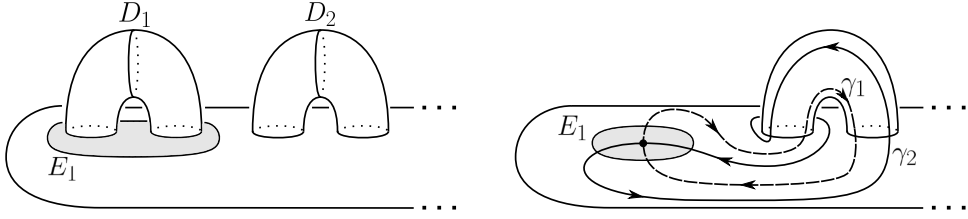}
\caption{Disk $E_1$ in $H_g$ and loops $\gamma _1$ and $\gamma _2$.}\label{composite_rel1}
\end{figure}

\subsection{A Condition for conjugations of genus-1 HBP-maps in the handlebody group}\label{section_conjugation}

In this section, we give a necessary and sufficient condition that a genus-1 HBP-map is conjugate to $t_{C_1}t_{C_2}^{-1}$ in $\mathcal{H}_{g,1}$. For proper disks $d_1, d_2, \dots , d_g$ in $H_g-{\rm int}D_0$, the pair $\{ d_1, d_2, \dots , d_g\}$ is a {\it meridian disk system} if each $d_i$ $(i=1,2,\dots ,g)$ is non-separating and we obtain a 3-ball by cutting $H_g$ along $d_1\sqcup d_2\sqcup  \cdots \sqcup d_g$. For example, $\{ \overline{D}_1, \overline{D}_2, \dots , \overline{D}_g\}$ is a meridian disk system, where $\overline{D}_1, \overline{D}_2, \dots , \overline{D}_g$ are disks in $H_g$ whose boundary components are $D_1, D_2, \dots , D_g$ as in Figure~\ref{handlebody_scc1}, respectively. Then we have the following proposition.

\begin{prop}\label{conj_hbpmap}
Let $\{ c_1,c_2\}$ be a genus-1 HBP on $\Sigma _{g,1}$. Then $t_{c_1}t_{c_2}^{-1}$ is conjugate to $t_{C_1}t_{C_2}^{-1}$ in $\mathcal{H}_{g,1}$ if and only if there exist a properly embedded annulus $A$ in $H_g$ whose boundary is $c_1\sqcup c_2$ and a meridian disk system $\{ d_1, d_2, \dots , d_g\}$ such that $d_2, \dots , d_g$ are disjoint from $A$ and the intersection of $d_1$ and $A$ is an arc which doesn't separate $A$.
\end{prop}

\begin{proof}
We suppose that $t_{c_1}t_{c_2}^{-1}$ is conjugate to $t_{C_1}t_{C_2}^{-1}$ in $\mathcal{H}_{g,1}$. Then there exists a diffeomorphism $f:H_g\rightarrow H_g$ such that the restriction $f|_{D_0}$ is identity map on $D_0$ and $f(c_i)=C_i$ $(i=1,2)$. By Figure~\ref{handlebody_scc1}, there exists a properly embedded annulus $A_0$ in $H_g$ whose boundary is $C_1\sqcup C_2$ such that the intersection of $\overline{D}_2$ and $A_0$ is an arc which doesn't separate $A_0$ and $\overline{D}_1, \overline{D}_3, \dots , \overline{D}_g$ are disjoint from $A_0$. Thus $A:=f(A_0)$, $d_1:=f(\overline{D}_2)$, $d_2:=f(\overline{D}_1)$, $d_3:=f(\overline{D}_3)$, $\dots , d_g:=f(\overline{D}_g)$ satisfy the condition above. We have proved the ``only if'' part of the proposition.

We suppose that there exist a properly embedded annulus $A$ in $H_g$ whose boundary is $c_1\sqcup c_2$ and a meridian disk system $\{ d_1, d_2, \dots , d_g\}$ such that $d_2, \dots , d_g$ are disjoint from $A$ and the intersection of $d_1$ and $A$ is an arc $\delta $ which doesn't separate $A$. Note that the arc $\delta $ separates $d_1$ into two disks $e^\prime $ and $e^{\prime \prime }$ in $H_g$. Let $B$ be a 3-ball which is obtained by cutting $H_g$ along $d_1\sqcup d_2\sqcup \cdots \sqcup d_g$. Since $\delta $ doesn't separate $A$, the image $\overline{A}$ of $A$ in $B$ is a proper disk in $B$. Hence $\overline{A}$ separates $B$ into 3-balls $B^\prime $ and $B^{\prime \prime }$. Without loss of generality, we can assume that the copies $e_1^\prime $ and $e_2^\prime $ of $e^\prime $ and copies $d_{i_0,1}$ and $d_{i_0,2}$ of $d_{i_0}$ are included in $B^\prime $ for some $i_0\in \{ 2,\dots ,g\}$, and the copies $e_1^{\prime \prime }$ and $e_2^{\prime \prime }$ of $e^{\prime \prime }$ and the copies $d_{i,1}$ and $d_{i,2}$ of $d_i$ are included in $B^{\prime \prime }$ for any $i\in \{ 2,\dots ,g\} -\{ i_0\}$ since $\{ c_1,c_2\}$ is a genus-1 HBP on $\Sigma _{g,1}$. Denote by $\overline{A}^\prime $ and $\overline{A}^{\prime \prime }$ the images of $\overline{A}$ in $B^\prime $ and $B^{\prime \prime }$, respectively. 

Let $V^\prime $ and $V^{\prime \prime }$ be the handlebodies which obtained by cutting $H_g$ along $A_0$ such that $V^\prime $ is diffeomorphic to $H_2$ and $V^{\prime \prime }$ is diffeomorphic to $H_{g-1}$ and let $B_0$, $B_0^\prime $ and $B_0^{\prime \prime }$ be the 3-balls which obtained by cutting $H_g$, $V^\prime $ and $V^{\prime \prime }$ along $\overline{D}_1\sqcup \cdots \sqcup \overline{D}_g$, respectively. Denote by $\overline{A_0}^\prime $ and $\overline{A_0}^{\prime \prime }$ the disks on $\partial B_0^\prime $ and $\partial B_0^{\prime \prime }$ which are obtained from $A$, by $\overline{D}_{j,1}$ and $\overline{D}_{j,2}$ are copies of disk $\overline{D}_j$ on $\partial B_0$ for $j\in \{ 1,\dots ,g\}$ and by $e_{0,k}^\prime $ and $e_{0,k}^{\prime \prime }$ the disks on $\partial B_0^\prime $ and $\partial B_0^{\prime \prime }$ which are obtained from $\overline{D}_{2,k}$ by cutting $\overline{D}_{2,k}$ along $A_0$ for $i\in \{ 1,2\}$, respectively. For $j\in \{ 1,3,\dots ,g\}$ and $k\in \{ 1,2\}$, we regard $\overline{D}_{j,k}$ as a disk in $\partial B_0^\prime \sqcup \partial B_0^{\prime \prime }$.
Since the isotopy classes of $\overline{A}^\prime $ and $\overline{A}^{\prime \prime }$ in $B^\prime $ and $B^{\prime \prime }$ (resp. $\overline{A_0}^\prime $ and $\overline{A_0}^{\prime \prime }$ in $B_0^\prime $ and $B_0^{\prime \prime }$) fixed $e_k^\prime $, $e_k^{\prime \prime }$ and $d_{i,k}$ ($i\in \{ 2,\dots ,g\}$, $k\in \{ 1,2\}$) (resp. $e_{0,k}^\prime $, $e_{0,k}^{\prime \prime }$ and $\overline{D}_{j,k}$ ($j\in \{ 1,3,\dots ,g\}$, $k\in \{ 1,2\}$)) depend on the isotopy classes of arcs which obtained from the center line of $A$ (resp. $A_0$), there exist orientation preserving diffoemorphisms $f^\prime :B^\prime \rightarrow B_0^\prime$ and $f^{\prime \prime }:B^{\prime \prime }\rightarrow B_0^{\prime \prime }$ such that $f^\prime (\overline{A}^\prime )=\overline{A_0}^\prime $, $f^{\prime \prime }(\overline{A}^{\prime \prime })=\overline{A_0}^{\prime \prime }$, the restriction $f^{\prime \prime }|_{D_0}$ is the identity map and $f^\prime $ and $f^{\prime \prime }$ are compatible with regluing of $B^\prime , B^{\prime \prime }, B_0^\prime , B_0^{\prime \prime }$ along $d_{i,k}$, $e_k^\prime $, $e_k^{\prime \prime }$, $\overline{D}_{j,k}$, $e_{0,k}^\prime $ and $e_{0,k}^{\prime \prime }$ for $i\in \{ 2,\dots ,g\}$, $j\in \{ 1,3,\dots ,g\}$ and $k\in \{ 1,2\}$. Such diffeomorphisms induce the diffeomorphism $\tilde{f}:H_g\rightarrow H_g$ such that $\tilde{f}(A)=A_0$ and $\tilde{f}|_{D_0}={\rm id}_{D_0}$. Thus $t_{c_1}t_{c_2}^{-1}$ is conjugate to $t_{C_1}t_{C_2}^{-1}$ in $\mathcal{H}_{g,1}$ and we have completed the proof of this proposition.
\end{proof}

Let $C_1^m$ and $C_2^m$ be simple closed curves on $\Sigma _{g,1}$ as in Figure~\ref{bpmap_e1e2} for $m\geq 2$. Since the union $C_1^m\sqcup C_2^m$ bounds an annulus $A_m$ in $H_g$ which intersects with $\overline{D}_1$ at $m$ proper arcs in $\overline{D}_1$ as in Figure~\ref{bpmap_e1e2}, $\{ C_1^m,C_2^m\}$ is a genus-1 HBP on $\Sigma _{g,1}$. Note that such an annulus is unique up to isotopy by the irreducibility of $H_g$. Then we show that $t_{C_1^m}t_{C_2^m}^{-1}$ is not conjugate to $t_{C_1}t_{C_2}^{-1}$ in $\mathcal{H}_{g,1}$ by Proposition~\ref{conj_hbpmap} and the next proposition.

\begin{prop}\label{hbp_notconj}
For $m\geq 2$, there does not exist a proper disk $D$ in $H_g$ which transversely intersects with $A_m$ at a proper arc in $D$ and separates $A_m$ into a disk.
\end{prop}

\begin{proof}
Suppose that there exists a proper disk $D$ in $H_g$ which transversely intersects with $A_m$ at a proper arc in $D$ and separates $A_m$ into a disk. Denote by $\delta $ the proper arc in $D$. For proper disks $d$ and $d^\prime $ in $H_g$ whose intersection is disjoint union of proper arcs in $d^\prime $, we obtain disks $\tilde{d}_1, \tilde{d}_2, \dots ,\tilde{d}_n$ in $H_g$ from the disk $d$ by cutting $d$ along $d^\prime $. Then there exist disks $e_1, e_2, \dots ,e_n$ in $d^\prime $ such that $\tilde{d}_1\cup e_1, \tilde{d}_2\cup e_2, \dots ,\tilde{d}_n\cup e_n$ are proper disks in $H_g$ and each $\tilde{d}_i\cup e_i$ $(i=1,\dots ,n)$ is isotopic to a proper disk $d_i$ in $H_g$ which doesn't intersect with $d^\prime $ and the other $d_j$. We call the operation which gives disjoint disks $d_1,d_2,\dots ,d_n$ from the disk $d$ the {\it surgery} on $d$ along $d^\prime $.

By the irreducibility of $H_g$, we can assume that the intersection of $D$ and $\overline{D}_2\sqcup \dots \sqcup \overline{D}_g$ is a disjoint union of proper arcs in $\overline{D}_2\sqcup \dots \sqcup \overline{D}_g$. Let $d_1,d_2,\dots ,d_n$ be proper disks in $H_g$ which are obtained from $D$ by the surgery on $D$ along $\overline{D}_2, \dots , \overline{D}_g$ and let $V$ be the solid torus which obtained from $H_g$ by cutting $H_g$ along $\overline{D}_2, \dots , \overline{D}_g$. Since $d_1,d_2,\dots ,d_n$, $\overline{D}_1$, $A_m$ and $\delta $ don't intersect $\overline{D}_2, \dots , \overline{D}_g$, we regard $d_1,d_2,\dots ,d_n$, $\overline{D}_1$, $A_m$ and $\delta $ as proper disks, a proper annulus and a proper arc in $V$. Note that the intersection of $A_m$ and $\overline{D}_1$ in $V$ is not a single arc up to ambient isotopy of $V$ (see Figure~\ref{solidtorus1}). Then there exists $i_0\in \{ 1, 2, \dots , n\}$ such that the proper disk $d_{i_0}$ in $V$ intersects with $A_m$ at the arc $\delta $. Since $\partial d_{i_0}\subset \partial V$ transversely intersects with each $C_k^m$ $(k=1,2)$ at one point, $d_{i_0}$ is a non-separating disk in $V$. Hence $d_{i_0}$ is isotopic to $\overline{D}_1$ in $V$ by forgetting the copies of $\overline{D}_2\sqcup \dots \sqcup \overline{D}_g$ throughout the isotopy. This is a contradiction to the fact that the intersection of $A_m$ and $\overline{D}_1$ in $V$ is not a single arc. We have completed the proof of this proposition.
\end{proof}

\begin{figure}[h]
\includegraphics[scale=0.90]{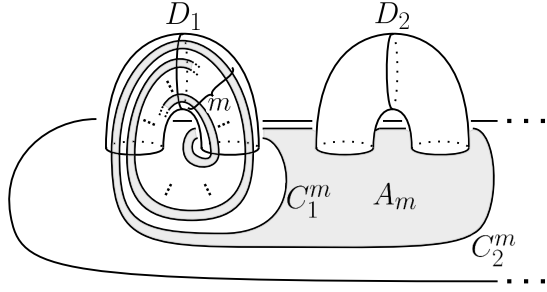}
\caption{Simple closed curves $C_1^m$ and $C_2^m$ in $\Sigma _{g,1}$ which bound an annulus $A_m$ in $H_g$.}\label{bpmap_e1e2}
\end{figure}

\begin{figure}[h]
\includegraphics[scale=0.90]{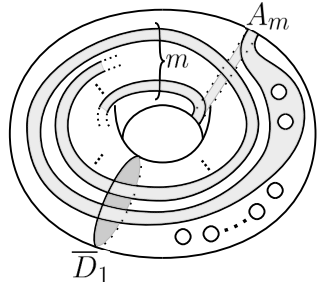}
\caption{Annulus $A_m$ and the disk $\overline{D}_1$ in the solid torus $V$. We express the copies of $\overline{D}_2\sqcup \dots \sqcup \overline{D}_g$ by the holes on $\partial V$.}\label{solidtorus1}
\end{figure}

\section{Applications}\label{application}

In this section, we prove Theorem~\ref{thm_application}. 

\subsection{Proof of Theorem~\ref{thm_application}}\label{proof_application}
Take a symplectic basis $\{ a_1,\dots ,a_g, b_1,\dots , b_g\}$ for ${\rm H}_1(\Sigma _{g,1};\mathbb Z)$ as in Figure~\ref{handlebody_basis1}. The symplectic group is ${\rm Sp}(2g,\mathbb Z)=\{ X\in {\rm GL}(2g,\mathbb Z)\mid {}^t\!XJ_{2g}X=J_{2g}\}$, where $J_{2g}=\left(
    \begin{array}{cc}
      0 & I_g  \\
      -I_g & 0 
    \end{array}
  \right)$ and $I_g$ is the identity matrix of rank $g$. We define
\begin{eqnarray*}
 {\rm urSp}(2g) &:=&\left\{ \left(
    \begin{array}{cc}
      A & B  \\
      C & D 
    \end{array}
  \right) \in {\rm GL}(2g,\mathbb Z)\middle| C=0\right\}\cap {\rm Sp}(2g,\mathbb Z)\\
&=& \left\{ \left(
    \begin{array}{cc}
      A & B  \\
      0 & {}^t\!A^{-1} 
    \end{array}
  \right) \middle|
\begin{array}{l}
\text{$A$ is unimodular,} \\
\text{$A^{-1}B$ is symmetric}
\end{array}\right\}. 
\end{eqnarray*}
The notation ${\rm urSp}(2g)$ was introduced by Hirose~\cite{Hirose}. The last equation and the next lemma is obtained from an argument in Section~2 of \cite{Birman1}. Recall the homomorphism $\Psi :\mathcal{M}_{g,1}\rightarrow {\rm Sp}(2g,\mathbb Z)$ induced by the action of $\mathcal{M}_{g,1}$ on ${\rm H}_1(\Sigma _{g,1};\mathbb Z)$.
\begin{lem}[\cite{Birman1}]\label{birman}
$\Psi (\mathcal{H}_{g,1})={\rm urSp}(2g)$.
\end{lem}

\begin{figure}[h]
\includegraphics[scale=0.80]{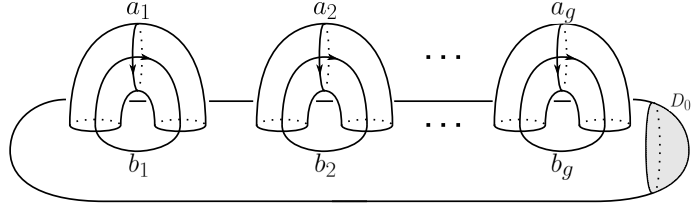}
\caption{Basis for the first homology group of $\Sigma _{g,1}$.}\label{handlebody_basis1}
\end{figure}

We review the next well-known lemma. 
\begin{lem}\label{gen_ker_composition}
Let $G$, $H$ and $Q$ be groups and let $\varphi :G\rightarrow H$ and $\psi :H\rightarrow Q$ be homomorphisms. We take a generating set $X$ for ${\rm ker}\psi |_{\varphi (G)}\subset H$ and a lift $\widetilde{X}\subset G$ of $X$ with respect to $\varphi $. Then ${\rm ker}\psi \circ \varphi $ is generated by ${\rm ker}\varphi $ and $\widetilde{X}$.
\end{lem}

Let $\Phi _d:{\rm Sp}(2g,\mathbb Z)\rightarrow {\rm Sp}(2g,\mathbb Z_d)$ be the homomorphism induced by the natural projection $\mathbb Z\rightarrow \mathbb Z_d$ for $d\geq 2$. Then we define 
\[
{\rm urSp}(2g)[d]:={\rm ker}\Phi _d|_{{\rm urSp}(2g)}\subset {\rm urSp}(2g).
\]
For distinct $1\leq i,j\leq g$, denote by $\mathcal{E}_{i,j}$ the $(g\times g)$-matrix whose $(i,j)$-entry is $1$ and the other entries are $0$, by $S_{i,j}$ the $(g\times g)$-matrix whose $(i,j)$-entry and $(j,i)$-entry are $1$ and the other entries are $0$ and by $S_{i,i}$ the $(g\times g)$-matrix whose $(i,i)$-entry is $1$ and the other entries are $0$. Then we define $E_{i,j}:=I_g+\mathcal{E}_{i,j}$, $F_i:=I_g-2S_{i,i}$ for distinct $1\leq i,j\leq g$ and 
\begin{eqnarray*}
X_{i,j}&:=&\left(
    \begin{array}{cc}
      E_{i,j} & 0  \\
      0  & -E_{j,i} 
    \end{array}
  \right) \text{ for distinct }1\leq i,j\leq g,\\
Y_{i,j}&:=& \left(
    \begin{array}{cc}
      I_g & S_{i,j}  \\
      0  & I_g 
    \end{array}
  \right) \text{ for }1\leq i,j\leq g,\\
Z_i&:=&\left(
    \begin{array}{cc}
      F_i & 0  \\
      0  & F_i 
    \end{array}
  \right) \text{ for }1\leq i\leq g.\\
\end{eqnarray*} 
Note that $X_{i,j}$ and $Y_{i,j}$ are elements of ${\rm urSp}(2g)$, $Z_i$ is an element of ${\rm urSp}(2g)[2]$, and $X_{i,j}^d$ and $Y_{i,i}^d$ are elements of ${\rm urSp}(2g)[d]$ for $d\geq 2$. Then we have the following proposition.
\begin{prop}\label{ursp}
For $g\geq 1$, ${\rm urSp}(2g)[2]$ is normally generated in ${\rm urSp}(2g)$ by $Y_{1,1}^2$ and $Z_1$.

For $g\geq 3$ and $d\geq 3$, ${\rm urSp}(2g)[d]$ is normally generated in ${\rm urSp}(2g)$ by $X_{1,2}^d$ and $Y_{1,1}^d$.
\end{prop}
We prove Proposition~\ref{ursp} in Section~\ref{proof_ursp}.
\begin{proof}[Proof of Theorem~\ref{thm_application}]
By the definition of $\mathcal{H}_{g,1}[d]$, $\mathcal{H}_{g,1}[d]$ is the kernel of the composition of $\Psi :\mathcal{H}_{g,1}\rightarrow {\rm Sp}(2g,\mathbb Z)$ and $\Phi _d:{\rm Sp}(2g,\mathbb Z)\rightarrow {\rm Sp}(2g,\mathbb Z_d)$. We apply Lemma~\ref{gen_ker_composition} to these homomorphisms. Since $\Psi (\mathcal{H}_{g,1})={\rm urSp}(2g)$, by Lemma~\ref{birman}, $\mathcal{H}_{g,1}[d]$ is generated by ${\rm ker}\Psi =\mathcal{IH}_{g,1}$ and a lift of a generating set for ${\rm ker}\Phi _d|_{{\rm urSp}(2g)}={\rm urSp}(2g)[d]$. We can check $\Psi (\alpha )=X_{1,2}$, $\Psi (t_{D_1})=Y_{1,1}$, $\Psi (\omega )=Z_1$ and conjugations of $X_{1,2}$, $Y_{1,1}$ and $Z_1$ in ${\rm urSp}(2g)$ lift conjugations of $\alpha $, $t_{D_1}$ and $\omega $ in $\mathcal{H}_{g,1}$ with respect to $\Psi $. Therefore, by Proposition~\ref{ursp}, $\mathcal{H}_{g,1}[2]$ is normally generated in $\mathcal{H}_{g,1}$ by $\omega $, $t_{D_1}^2$ and a genus-1 HBP-map, and $\mathcal{H}_{g,1}[d]$ is normally generated in $\mathcal{H}_{g,1}$ by $\alpha ^d$, $t_{D_1}^d$ and a genus-1 HBP-map for $g\geq 3$ and $d\geq 3$. We have completed the proof of Theorem~\ref{thm_application}. 
\end{proof}

\subsection{A normal generating set for $\boldsymbol{{\rm urSp}(2g)[d]}$}\label{proof_ursp}
In this section, we give a proof of Proposition~\ref{ursp}. The {\it level $d$ principal congruence subgroup $\Gamma _d(g)$ (resp. ${\rm SL}(g,\mathbb Z)[d]$) of ${\rm GL}(g,\mathbb Z)$ (resp. ${\rm SL}(g,\mathbb Z)$)} is the kernel of the natural homomorphism ${\rm GL}(g,\mathbb Z)\rightarrow {\rm GL}(g,\mathbb Z_d)$ (resp. ${\rm SL}(g,\mathbb Z)\rightarrow {\rm SL}(g,\mathbb Z_d)$). For $g\geq 1$, $\Gamma _2(g)$ is generated by $E_{i,j}^2$ and $F_i$ for distinct $1\leq i,j\leq g$ (see for instance \cite{Mccarthy-Pinkall}). In particular, we have the following lemma.
\begin{lem}\label{normalgen_level2_gl}
For $g\geq 1$, $\Gamma _2(g)$ is normally generated in $GL(g,\mathbb Z)$ by $F_1$.
\end{lem}
To prove Lemma~\ref{normalgen_level2_gl}, we prepare the following easy lemma.
\begin{lem}\label{conj_eij}
For distinct $1\leq i,j\leq g$, each $E_{i,j}$ is conjugate to $E_{1,2}$ in $GL(g,\mathbb Z)$.
\end{lem}
\begin{proof}[proof of Lemma~\ref{normalgen_level2_gl}]
Since
\begin{eqnarray*}
\left(
\begin{array}{cc}
      0 & 1  \\
      1 & 0 
\end{array}
\right) 
\left(
\begin{array}{cc}
      1 & 0  \\
      0 & -1 
\end{array}
\right)
\left(
\begin{array}{cc}
      0 & 1  \\
      1 & 0 
\end{array}
\right) 
=
\left(
\begin{array}{cc}
      -1 & 0  \\
      0 & 1 
\end{array}
\right) ,   
\end{eqnarray*} 
each $F_i$ is conjugate to $F_1$ in ${\rm GL}(g,\mathbb Z)$. By Lemma~\ref{conj_eij}, it is enough for the proof of Lemma~\ref{normalgen_level2_gl} to show that $E_{1,2}^2$ is a product of conjugations of $F_1$ in ${\rm GL}(g,\mathbb Z)$. Since 
\begin{eqnarray*}
\left(
\begin{array}{cc}
      1 & 1  \\
      0 & 1 
\end{array}
\right) 
\left(
\begin{array}{cc}
      -1 & 0  \\
      0 & 1 
\end{array}
\right)
\left(
\begin{array}{cc}
      1 & -1  \\
      0 & 1 
\end{array}
\right) 
\cdot \left(
\begin{array}{cc}
      -1 & 0  \\
      0 & 1 
\end{array}
\right)
=
\left(
\begin{array}{cc}
      1 & 2  \\
      0 & 1 
\end{array}
\right) , 
\end{eqnarray*} 
we have $E_{1,2}F_1E_{1,2}^{-1}\cdot F_1=E_{1,2}^2$. Therefore we get Lemma~\ref{normalgen_level2_gl}.
\end{proof}

We note that $\Gamma _d(g)={\rm SL}(g,\mathbb Z)[d]$ for $d\geq 3$. Bass-Milnor-Serre~\cite{Bass-Milnor-Serre} gave a generating set for ${\rm SL}(g,\mathbb Z)[d]$ as follows.
\begin{thm}[\cite{Bass-Milnor-Serre}, see also \cite{Putman3}]\label{Bass-Milnor-Serre}
For $g\geq 3$ and $d\geq 3$, ${\rm SL}(g,\mathbb Z)[d]=\Gamma _d(g)$ is normally generated in ${\rm SL}(g,\mathbb Z)$ by $E_{i,j}^d$ for distinct $1\leq i,j\leq g$.
\end{thm}
By Lemma~\ref{conj_eij} and Theorem~\ref{Bass-Milnor-Serre}, we have the following lemma.
\begin{lem}\label{normalgen_leveld_gl}
For $g\geq 3$ and $d\geq 3$, $\Gamma _d(g)$ is normally generated in ${\rm GL}(g,\mathbb Z)$ by $E_{1,2}^d$.
\end{lem}

We define the normal subgroup
\begin{eqnarray*}
 \mathcal{S}_g := \left\{ \left(
    \begin{array}{cc}
      I_g & B  \\
      0 & I_g 
    \end{array}
  \right) \middle|
\begin{array}{l}
\text{$B$ is symmetric}
\end{array}\right\}
\end{eqnarray*}
of {\rm urSp}(2g) and the kernel $\mathcal{S}_g[d]$ of the homomorphism $\Phi _d|_{\mathcal{S}_g}:\mathcal{S}_g\rightarrow {\rm Sp}(2g,\mathbb Z_d)$. Note that each $Y_{i,j}$ is an element of $\mathcal{S}_g$, each $Y_{i,j}^d$ is an element of $\mathcal{S}_g[d]$ and $\mathcal{S}_g$ is an abelian group since 
\[
\left(
\begin{array}{cc}
      I_g & B  \\
      0 & I_g 
\end{array}
\right) 
\left(
\begin{array}{cc}
      I_g & B^\prime  \\
      0 & I_g 
\end{array}
\right)
=
\left(
\begin{array}{cc}
      I_g & B+B^\prime  \\
      0 & I_g 
\end{array}
\right) .
\]
We have the following lemma.
\begin{lem}\label{normalgen_sg}
For $g\geq 1$, $\mathcal{S}_g$ is normally generated in ${\rm urSp}(2g)$ by $Y_{1,1}$.
\end{lem}
Since $\mathcal{S}_g$ is abelian and $\mathcal{S}_g[d]$ is generated by $Y_{i,j}^d$ for $1\leq i,j\leq g$, we have the following corollary of Lemma~\ref{normalgen_sg}.
\begin{cor}\label{normalgen_leveld_sg}
For $g\geq 1$ and $d\geq 2$, $\mathcal{S}_g[d]$ is normally generated in ${\rm urSp}(2g)$ by $Y_{1,1}^d$.
\end{cor}
\begin{proof}[Proof of Lemma~\ref{normalgen_sg}]
Since $\mathcal{S}_g$ is generated by $Y_{i,j}$ for $1\leq i,j\leq g$, it is enough for the proof of Lemma~\ref{normalgen_sg} to show that each $Y_{i,j}$ is a product of conjugations of $Y_{1,1}$ in ${\rm urSp}(2g)$. Note that
\begin{eqnarray*}
\left(
    \begin{array}{cc}
      A & 0  \\
      0 & {}^t\!A^{-1} 
    \end{array}
  \right)
\left(
\begin{array}{cc}
      I_g & B  \\
      0 & I_g 
\end{array}
\right) 
\left(
    \begin{array}{cc}
      A & 0  \\
      0 & {}^t\!A^{-1} 
    \end{array}
  \right)^{-1}
&=&
\left(
    \begin{array}{cc}
      I_g & AB{}^t\!A  \\
      0 & I_g 
    \end{array}
  \right).
\end{eqnarray*}
We define $A_{i,j}:=I_g+S_{i,j}-S_{i,i}-S_{j,j}\in GL(g,\mathbb Z)$ and 
\begin{eqnarray*}
\widetilde{A}_{i,j}:=
\left(
    \begin{array}{cc}
      A_{i,j} & 0 \\
      0 & A_{i,j} 
    \end{array}
  \right) \in {\rm urSp}(2g)
\end{eqnarray*}
for distinct $1\leq i,j\leq g$. We remark that $A_{i,j}={}^t\!A_{i,j}=A_{i,j}^{-1}$ and $\widetilde{A}_{i,j}=\widetilde{A}_{i,j}^{-1}$. Since $A_{1,i}S_{i,j}A_{1,i}=S_{1,j}$ and $A_{2,j}S_{1,j}A_{2,j}=S_{1,2}$ for distinct $1\leq i,j\leq g$, we have $\widetilde{A}_{1,i}Y_{i,j}\widetilde{A}_{1,i}=Y_{1,j}$ and $\widetilde{A}_{2,j}Y_{1,j}\widetilde{A}_{2,j}=Y_{1,2}$. Hence each $Y_{i,j}$ is conjugate to $Y_{1,2}$ in ${\rm urSp}(2g)$ for distinct $1\leq i,j\leq g$.

Since $A_{1,i}S_{i,i}A_{1,i}=S_{1,1}$, we have $\widetilde{A}_{1,i}Y_{i,i}\widetilde{A}_{1,i}=Y_{1,1}$ for $1\leq i\leq g$. Thus it is enough for the proof of Lemma~\ref{normalgen_sg} to show that each $Y_{1,2}$ is a product of conjugations of $Y_{1,1}$ in ${\rm urSp}(2g)$. We can check $Y_{1,1}^{-1}\cdot X_{2,1}Y_{1,1}X_{2,1}^{-1}\cdot Y_{2,2}^{-1}=Y_{1,2}$ and we get Lemma~~\ref{normalgen_sg}.
\end{proof}

\begin{proof}[Proof of Proposition~\ref{ursp}]
For each $X=\left(
    \begin{array}{cc}
      A & B  \\
      0 & {}^t\!A^{-1} 
    \end{array}
  \right)\in {\rm urSp}(2g)[d]$, $A$ is unimodular and $A\equiv I_g$ modulo $d$. The condition means $A\in \Gamma _d(g)$. For $g\geq 3$ (resp. $g=2$), by Lemma~\ref{normalgen_leveld_gl} (resp. Lemma~\ref{normalgen_level2_gl}), there exists a product $X^\prime $ of conjugations of $E_{1,2}^d$ (resp. $F_1$) in $\Gamma _d(g)$ such that $A=X^\prime $. Then $\widetilde{X^\prime }:=\left(
    \begin{array}{cc}
      X^\prime  & 0  \\
      0 & {}^t\!(X^\prime )^{-1} 
    \end{array}
  \right)\in {\rm urSp}(2g)[d]$ is a product of conjugations of $X_{1,2}^d$ (resp. $Z_1$) in ${\rm urSp}(2g)$ for $d\geq 3$ (resp. $d=2$). Since $A(X^\prime )^{-1}=I_g$, $X\widetilde{X^\prime }^{-1}$ is an element of $\mathcal{S}_g[d]$. By Corollary~\ref{normalgen_leveld_sg}, there exist a product $Y$ of conjugations of $Y_{1,1}^d$ in ${\rm urSp}(2g)$ such that $X\widetilde{X^\prime }^{-1}=Y$. We have $X=Y\widetilde{X^\prime }$ and we have completed the proof of Proposition~\ref{ursp}. 
\end{proof}

\section{Proof of Corollaries}

In this section, we prove Corollary~\ref{thm_cor} and Corollary~\ref{thm_cor2}. 

\subsection{Proof of Corollary~\ref{thm_cor}}~\label{proof_cor}
For $d\geq 2$, we define 
\begin{eqnarray*}
 {\rm urSp}(2g,\mathbb Z_d) &:=&\left\{ \left(
    \begin{array}{cc}
      A & B  \\
      0 & {}^t\!A^{-1} 
    \end{array}
  \right) \in {\rm Sp}(2g,\mathbb Z_d)\middle|
\begin{array}{l}
\text{$A$ is unimodular,} \\
\text{$A^{-1}B$ is symmetric}
\end{array}\right\},\\
\mathcal{S}_g(d) &:=& \left\{ \left(
    \begin{array}{cc}
      I_g & B  \\
      0 & I_g 
    \end{array}
  \right) \in {\rm Sp}(2g,\mathbb Z_d)\middle|
\begin{array}{l}
\text{$B$ is symmetric}
\end{array}\right\}. 
\end{eqnarray*}
For convenience, we define $\mathcal{H}_{g,1}[1]:=\mathcal{IH}_{g,1}$, $\Gamma _1(H_g\text{ rel }D_0):=\mathcal{I}(H_g\text{ rel }D_0)$, ${\rm urSp}(2g,\mathbb Z_1):={\rm urSp}(2g)$, $\mathcal{S}_g(1):=\mathcal{S}_g$ and $\Psi _1:=\Psi $. By an argument similar to that in Section~2 of \cite{Birman1}, Lemma~\ref{birman} is generalized into the following lemma.
\begin{lem}
For $d\geq 1$, $\Psi _d(\mathcal{H}_{g,1})={\rm urSp}(2g,\mathbb Z_d)$.
\end{lem}

\begin{proof}[Proof of Corollary~\ref{thm_cor}]
Assume $d\geq 1$ and $g\geq 3$. For $f\in \Gamma _d(H_g\text{ rel }D_0)$), by the definition, $\Psi _d(f)$ is an element of $\mathcal{S}_g(d)$. Since $t_{D_1}$ is an element of $\Gamma _d(H_g\text{ rel }D_0)$ and $\mathcal{S}_g(d)$ is normally generated in ${\rm urSp}(2g,\mathbb Z_d)$ by $\Phi _d(Y_{1,1})$ by an argument similar to that in the proof of Lemma~\ref{normalgen_sg}, we have $\Psi _d(\Gamma _d(H_g\text{ rel }D_0))=\mathcal{S}_g(d)$. Hence we have the exact sequence
\begin{eqnarray*}
1\longrightarrow \mathcal{H}_{g,1}[d]\longrightarrow \Gamma _d(H_g\text{ rel }D_0)\stackrel{\Psi _d|_{\Gamma _d(H_g\text{ rel }D_0)}}{\longrightarrow }\mathcal{S}_g(d)\longrightarrow 1.
\end{eqnarray*}
By the exact sequence, $\Gamma _d(H_g\text{ rel }D_0)$ is generated by $\mathcal{H}_{g,1}[d]$ and conjugations of $t_{D_1}$ in $\mathcal{H}_{g,1}$. Therefore, by Theorem~\ref{mainthm} and Theorem~\ref{thm_application}, $\mathcal{I}(H_g\text{ rel }D_0)$ is normally generated in $\mathcal{H}_{g,1}$ by $t_{D_1}$ and $t_{C_1}t_{C_2}^{-1}$, $\Gamma _2(H_g\text{ rel }D_0)$ is normally generated in $\mathcal{H}_{g,1}$ by $\omega $, $t_{D_1}$ and $t_{C_1}t_{C_2}^{-1}$ and $\Gamma _d(H_g\text{ rel }D_0)$ is normally generated in $\mathcal{H}_{g,1}$ by $\alpha ^d$, $t_{D_1}$ and $t_{C_1}t_{C_2}^{-1}$ for $d\geq 3$. We have completed the proof of Corollary~\ref{thm_cor}.
\end{proof}

\subsection{Proof of Corollary~\ref{thm_cor2}}~\label{proof_cor2}

Let $D_2^{\prime \prime}$ be a simple closed curve on $\Sigma _{g,1}$ as in Figure~\ref{d2doubleprime}. Note that $D_2^{\prime \prime }$ bounds a disk in $H_g$.

\begin{proof}[Proof of Corollary~\ref{thm_cor2}]
By Corollary~\ref{thm_cor}, $\Gamma _2(H_g\text{ rel }D_0)$ is normally generated in $\mathcal{H}_{g,1}$ by $\omega $, $t_{D_1}$ and $t_{C_1}t_{C_2}^{-1}$ for $g\geq 3$. Hence it is sufficient for the proof of Corollary~\ref{thm_cor2} to prove that $t_{C_1}t_{C_2}^{-1}$ is a product of conjugations of $\omega $, $t_{D_1}$ and $\alpha ^2$ in $\mathcal{H}_{g,1}$. Recall that $\alpha =t_{C_1}t_{C_2^\prime }^{-1}$.

Define $f:=t_{D_2}t_{D_2^{\prime \prime }}^{-1}\omega ^{-1}\in \mathcal{H}_{g,1}$. We remark that $t_{D_2}$ and $t_{D_2^{\prime \prime }}$ are conjugate to $t_{D_1}$ in $\mathcal{H}_{g,1}$ since $D_2$ and $D_2^{\prime \prime }$ bound non-separating proper disks in $H_g$. We can check that $f(C_1)=C_2^\prime $ and $f(C_2^\prime )=C_2$. Then we have
\begin{eqnarray*}
t_{C_1}t_{C_2}^{-1}&=&t_{C_1}t_{C_2^\prime }^{-1}\cdot t_{C_2^\prime }t_{C_2}^{-1}\\
&=&\alpha \cdot f(t_{C_1}t_{C_2^\prime }^{-1})f^{-1}\\
&=&\alpha ^2\cdot \alpha ^{-1}f\alpha \cdot f^{-1}.
\end{eqnarray*}
We have completed the proof of Corollary~\ref{thm_cor2}.

\end{proof}

\begin{figure}[h]
\includegraphics[scale=0.7]{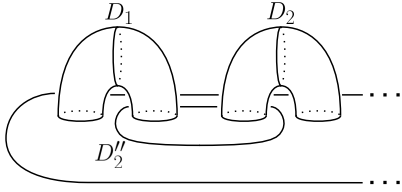}
\caption{Simple closed curve $D_2^{\prime \prime}$ on $\Sigma _{g,1}$.}\label{d2doubleprime}
\end{figure}

\par
{\bf Acknowledgements: } The author would like to express his gratitude to Hisaaki Endo, for his encouragement and helpful advices. The author also wishes to thank Susumu Hirose and Wolfgang Pitsch for their comments and helpful advices. The author was supported by JSPS KAKENHI Grant number 15J10066.

\end{document}